\documentclass[10pt, reqno]{amsart}
\date{\today}

\usepackage{enumerate}
\usepackage[margin=1.5in]{geometry}
\usepackage{hyperref}
\usepackage{tikz-cd}
\usepackage{quiver}
\usepackage{ulem}
\usepackage{booktabs}
\usepackage[nameinlink,capitalise,noabbrev]{cleveref}
\usepackage[shortalphabetic]{amsrefs}

\usetikzlibrary{positioning} 
\usetikzlibrary{decorations.pathreplacing}
\usetikzlibrary{arrows, decorations.markings} 
\pgfarrowsdeclarecombine{twotriang}{twotriang}{stealth}{stealth}
{stealth}{stealth}
\tikzset{mono/.style={>-stealth}} 
\tikzset{epi/.style={-twotriang}} 
\tikzset{arrow/.style={->}}
\tikzset{arrowshorter/.style={->, shorten <=2pt, shorten >=2pt}}
\tikzset{twoarrowlonger/.style={double,double distance=1.5pt,
shorten <=5pt,shorten >=6pt,
decoration={markings,mark=at position -4pt with {\arrow[scale=1.75]{>}}},
preaction={decorate}}} 

\tikzset{mapstikz/.style={-stealth, 
decoration={markings,mark=at position 0pt with {\arrow[scale=0.5]{|}}}, preaction={decorate}}}
\tikzset{dot/.style={circle,draw,fill,inner sep=1pt}}

\newcommand{\Q}{\mathbb{Q}}
\renewcommand{\S}{\mathbb{S}}
\newcommand{\Z}{\mathbb{Z}}
\newcommand{\C}{\mathbb{C}}

\newcommand{\uA}{\underline{A}}
\newcommand{\uJ}{\underline{J}}
\newcommand{\uM}{\underline{M}}
\newcommand{\upi}{\underline{\pi}}

\newcommand{\KU}{KU}
\newcommand{\KUG}{{\KU_G}}
\newcommand{\RU}{\underline{RU}}
\newcommand{\uRQ}{\underline{\RQ}}
\newcommand{\RQ}{R\Q}
\newcommand{\Sp}{\textnormal{Sp}}

\DeclareMathOperator{\coker}{coker} 
\DeclareMathOperator{\colim}{colim} 
\DeclareMathOperator{\res}{res} 
\DeclareMathOperator{\tr}{tr} 
\DeclareMathOperator{\Cl}{Cl}

\newcommand{\into}{\hookrightarrow}

\newcommand{\Mack}{\textnormal{Mack}}
\newcommand{\Mod}{\textnormal{Mod}}

\newcommand{\ssfrac}[2]{\mathchoice{\raisebox{5pt}{$#1$}\!\big/\!\raisebox{-5pt}{$#2$}}{{}^{#1}\!/_{\!#2}}{ #1 /  #2}{ #1 /  #2}}

\newcommand{\bigsfrac}[2]{\displaystyle \ssfrac{#1}{#2}}

\newcommand{\cokpsi}{\underline{\mathrm{cok}}}

\crefname{equation}{}{}

\theoremstyle{plain}
\newtheorem{theorem}[equation]{Theorem}
\newtheorem{corollary}[equation]{Corollary}
\newtheorem{proposition}[equation]{Proposition}
\newtheorem{lemma}[equation]{Lemma}

\theoremstyle{definition}

\newtheorem{example}[equation]{Example}
\newtheorem{remark}[equation]{Remark}
\newtheorem{notation}[equation]{Notation}

\numberwithin{equation}{section}

\begin{document}

\title{The  homotopy of the $KU_G$-local equivariant sphere spectrum}

\author[T. Carawan]{Tanner {N.} Carawan}
\address{Tanner N. Carawan}
\email{tc2bb@virginia.edu}
\author[R. E. Field]{Rebecca Field}
\address{Rebecca Field}
\email{fieldre@jmu.edu}
\author[B. J. Guillou]{Bertrand {J}. Guillou}
\address{Bertrand J. Guillou}
\email{bertguillou@uky.edu}
\author[D. Mehrle]{David Mehrle}
\address{David Mehrle}
\email{davidm@uky.edu}
\author[N. J. Stapleton]{Nathaniel {J}. Stapleton}
\address{Nathaniel J. Stapleton}
\email{nat.j.stapleton@uky.edu}
\thanks{
Guillou was supported by NSF grant DMS-2003204.
Stapleton was supported by NSF grant DMS-1906236 and a Sloan Fellowship.
This collaboration was made possible by NSF RTG grant DMS-1839968.
}

\maketitle

\begin{abstract}
We compute the  homotopy Mackey functors of the $\KU_G$-local equivariant sphere spectrum when $G$ is a finite $q$-group for an odd prime $q$, building on the degree zero case from  \cite{BGS}. 
\end{abstract}

\section{Introduction}

Work of Adams--Baird (unpublished) \cite{AB} 
and Ravenel \cite[Theorems 8.10, 8.15]{Rav}, (\emph{cf}.~\cite[Corollaries 4.5, 4.6]{Bousfield}) calculates the homotopy groups of the sphere spectrum localized at complex $K$-theory.

Both the sphere spectrum $\S$ and the  complex $K$-theory spectrum $\KU$ admit equivariant refinements. It is natural to ask if the calculation of Adams--Baird and Ravenel can be done equivariantly. When $q$ is an odd prime and $G$ is a finite $q$-group, Bonventre and the third and fifth authors calculated the homotopy Mackey functor $\upi_0 L_{\KU_G}\S_G$ of the localization of the equivariant sphere spectrum $\S_G$ with respect to equivariant complex $K$-theory $\KU_G$ \cite[Theorem 1.1]{BGS}. In this paper, we calculate the remaining homotopy Mackey functors of this localization: $\upi_n L_{\KU_G}\S_G$ for $n \neq 0$.

\begin{theorem}
	Let $G$ be a $q$-group for an odd prime $q$, and let $\ell$ be any integer that is primitive mod $|G|$. The homotopy Mackey functors of the $\KU_G$-local equivariant sphere spectrum $L_{\KU_G}\S_G$ are as in the table below.
\[
   \begin{array}{@{}l@{\hspace{1cm}}l@{}}
        \toprule
        		n 
			& 
				\upi_n L_{\KU_G} \S_G   
			\\
        \midrule 
	        	0 
			& 
				\uRQ \otimes (\Z \oplus \Z/2)^\dagger 
			\\ 
        		-2 
			& 
				\coker\big( \RU \xrightarrow{\psi^\ell - 1} \RU\big) \otimes \Q/\Z
			\\ 
                8k\neq 0  \textnormal{ or } 8k + 2
		    & 
		    	\displaystyle \uRQ \otimes \Z/2 
					\vphantom{\RU \xrightarrow{\psi^\ell-1}\RU} 
			\\
		        2k-1\neq -1
		    & 
		    	\uRQ \otimes \pi_{2k-1} L_{\KU} \S\big[\tfrac{1}{q}] 
						\oplus 
				\coker\big(\RU_q^\wedge\{\beta^k\} 
						\xrightarrow{\psi^\ell - 1 } 
						\RU_q^\wedge\{\beta^k\}\big) 
			\\
				\textnormal{otherwise}
			& 
				0 \vphantom{\RU \xrightarrow{\psi^\ell-1}\RU}
			\\
        \bottomrule
		\multicolumn{2}{l}{\footnotesize{\dagger \text{\cite[Theorem 1.1]{BGS}}}}
   \end{array}
\]
\end{theorem}

The structure of the cokernel of $\psi^\ell - 1$ is described in \cref{cokerpsiell}. 

\begin{remark}
	 These homotopy groups are independent of the choice of $\ell$, so long as it is primitive modulo $|G|$.
\end{remark}

To perform this computation, we make use of a fiber sequence of equivariant spectra involving Adams operations. For $G$ any odd $q$-group and  $\ell$ primitive mod $|G|$, there is a fiber sequence
	\begin{equation} 
    	L_{\KU_G/q}\S_G 
    		\longrightarrow 
    	(\KU_G)_q^\wedge 
    		\xrightarrow{\psi^\ell - 1} 
    	(\KU_G)_q^\wedge.
	\end{equation}
To use this to understand the homotopy Mackey functors of $L_{\KU_G/q}\S_G$, one must contend with the action of $\psi^\ell-1$ on the homotopy Mackey functors of $\KU_G$. Recall that $\upi_* \KU_G \cong \RU[\beta, \beta^{-1}]$, where $|\beta|=2$ and that $\psi^\ell$ is a ring map with $\psi^\ell(\beta) = \ell \beta$. It is not too hard to see that the kernel of $\psi^\ell-1 \colon \upi_{2k} \KU_G \to \upi_{2k} \KU_G$ is trivial for $k \neq 0$. Thus the question is really to understand the cokernel of $\psi^\ell-1$. We accomplish this in  \cref{sec:algebra}, where we give an explicit description of the values of the Mackey functor $\coker(\psi^\ell-1)$. With these in hand, we then use the arithmetic fracture square to obtain the global calculation in \cref{sec:homotopy}.

\subsection{Conventions}
\label{conventions}
Throughout this paper, we fix an odd prime $q$ and a finite $q$-group $G$. The integer $\ell$ will always be assumed to be coprime to the order of $G$, and at times $\ell$ will furthermore be assumed to be primitive modulo $|G|=q^j$. For $j > 2$, this is equivalent to requiring $j$ to be primitive modulo $q^2$ (see, for example, \cite{Narkiewicz}*{Lemma~1.10}).

\subsection{Acknowledgements}
This work began at an NSF RTG-funded workshop held in Lexington, Virginia in August 2022.
We would like to thank Julie Bergner, Nick Kuhn, and the other organizers 
of this workshop. We would also like to thank William Balderamma for numerous helpful discussions.

\section{Preliminaries}

\subsection{Representation rings and Green functors}

Recall the following commutative rings associated to $G$: 
\begin{itemize}
	\item The complex (resp.\ rational) representation ring $RU(G)$ (resp.\ $\RQ(G)$) is the Groth\-endieck group of isomorphism classes of finite-dimensional complex (resp.\ rational) $G$-representations under direct sum. The product is induced by the tensor product of $G$-representations. 
	\item The ring of complex-valued class functions $\Cl(G,\C)$ is the ring of functions $G \to \C$ which are constant on conjugacy classes of elements in $G$. 
\end{itemize} 
These rings are related by the following pair of ring homomorphisms: 
\[ 
	\begin{tikzcd}
		R\Q(G) \ar[r] & 
		RU(G) \ar[r, hook, "\chi"] & 
		\Cl(G,\C)
	\end{tikzcd}
\]
The first of these homomorphisms 
is base change from $\Q$ to $\C$, and the second is the character map. In particular, note that the character map $\chi \colon RU(G) \to \Cl(G,\C)$ is injective and embeds the complex representation ring as a subring of the ring of class functions. It will occasionally be convenient to calculate in the image of the character map rather than with complex representations themselves. 

These four commutative rings can all be upgraded to Green functors. We denote the Green functor versions of these rings by underlining them; for example, $\uRQ$ is the Green functor with $\uRQ(G/H) = \RQ(H)$. The same relationships hold among the Green functors as do among the commutative rings: there is a sequence of Green functor homomorphisms 
\[ 
	\begin{tikzcd}
		\uRQ \ar[r] & 
		\RU \ar[r, hook, "\chi"] & 
		\underline{Cl}
	\end{tikzcd}
\]

Let $A(G)$ be the Burnside ring of $G$, and let $\uA$ be the Burnside ring $G$-functor. There is a Green functor homomorphism $\uA \to \uRQ$ given levelwise by taking a finite $G$-set to the associated permutation representation. When $G$ is a $p$-group, the Ritter--Segal theorem \cite{Ritter,Seg-perm} says that $\uA \to \uRQ$ is surjective; we name its kernel $\uJ$ and deduce an isomorphism $\uA / \uJ \cong \uRQ$. 
The ideal $J(G)$ admits a nice description as the ideal of $A(G)$ generated by all virtual $G$-sets $X$ such that $|X^g| = 0$ for all $g \in G$.
The article \cite{BGS} uses the notation $\uA/\uJ$ throughout, but here we use the simpler notation $\uRQ$.

\subsection{Equivariant homotopy theory} 

Let $\Sp^G$ denote the category of genuine equivariant $G$-spectra. Examples include the $G$-equivariant sphere spectrum $\S_G$ and the $G$-spectrum of $G$-equivariant complex topological $K$-theory $\KU_G$. 

The homotopy of genuine $G$-spectra is naturally Mackey-functor valued. 
For the primary spectra in question in this paper, we have
\[
	\upi_0\S_G = \uA \quad \text{ and } \quad \upi_* \KU_G = \RU[\beta, \beta^{-1}] \text{ with } |\beta| = 2.
\]

If $E$ and $X$ are $G$-spectra, let $L_EX$ denote the Bousfield localization of $X$ at $E$. In particular, when $E = \S_G /p$ and $X$ is any spectrum, this localization is the $p$-completion of $X$, denoted by $X_p^\wedge := L_{\S_G/p} X$. If $X$ is already a localization $X = L_EY$, the localization of $X$ at $\S_G/p$ may be written 
\[
	L_{E/p} Y = L_{\S_G/p} L_E Y. 
\]
When $E = \S_G \wedge H\Q$ is the rational equivariant sphere, we obtain the rationalization of $X$, denoted $X \otimes \Q := L_{H_G \Q \otimes \uA} X$. 

The $p$-completion and the rationalization are related by a homotopy pullback square of $G$-spectra, called the arithmetic fracture square. 
When $X = L_{\KU_G}\S_G$, this is the square: 
\begin{equation}
\label{fracturesquare}
\begin{tikzcd}
	L_{KU_G}\S_G
	\ar[r]
	\ar[d]
	&
	\prod_pL_{KU_{G/p}}\S_G
	\ar[d]
	\\
	\Q\otimes L_{KU_G}\S_G
	\ar[r]
	&
	\Q\otimes\prod_pL_{KU_{G/p}}\S_G.
\end{tikzcd}
\end{equation} 
See \cite{TMF}*{Proposition~2.2 of Chapter~6} for a general version of the arithmetic square, from which \eqref{fracturesquare} can be deduced.
This is a useful tool for computing homotopy of $G$-spectra.

\section{The cokernel of $\psi^\ell -1$ acting on $\upi_*\KUG$}
\label{sec:algebra}

Recall that for $\ell \in \Z$ the Adams operation $\psi^\ell \colon \KU_G(X) \to \KU_G(X)$ is a ring homomorphism natural in the $G$-space $X$. In this section, we analyze $\psi^\ell - 1$ as a map on the complex 
representation ring of $G$ and on related objects. Recall \cref{conventions}: the 
 integer $\ell$ will always be assumed to be coprime to the order of $G$, and at times $\ell$ will furthermore be assumed to be primitive modulo $|G|=q^j$. 
 The following is Exercise~9.4 of \cite{Serre}.

\begin{lemma}
\label{AdamsOpsPermuteIrreps}
	The Adams operation $\psi^\ell \colon RU(G) \to RU(G)$ permutes the basis of irreducible representations if $\ell$ is coprime to $|G|$.
\end{lemma}

\begin{proof}
	Recall that a class function $\chi$ is the character of an irreducible representation if and only if $\chi(e)\geq 0$ and $\langle \chi, \chi \rangle=1$. 
	On a class function $f$, the Adams operation $\psi^\ell$ acts as $\psi^\ell(f)(g) = f(g^\ell)$.
	Since $\ell$ is coprime to $|G|$, every element has an $\ell$th root, so that the $\ell$th power determines a bijection  on $G$. 
	It follows that the Adams operation preserves the  inner product.
\end{proof}

\begin{lemma}\label{pcompletepsiell}
	Suppose that $\ell$ is coprime to $|G|$.
 	The Adams operation $\psi^\ell$ on $\RU$ is a homomorphism of Green functors.
\end{lemma}

\begin{proof}
The Adams operation $\psi^\ell$ is a levelwise ring homomorphism, and it is straightforward that $\psi^\ell$ commutes with restriction.  The main point is to show that $\psi^\ell$ commutes with induction of representations. To see this, we can use the character map to embed $\RU$ into the Green functor of class functions. As this is levelwise an injection, it suffices to see that $\psi^\ell$ commutes with induction for class functions. Here, the formula (see \cite{Serre}*{Section~7.2} is
\[
	\mathrm{Ind_H^G}(f)(g) = \frac1{|H|} \sum_{\substack{\gamma \in G, \\ \gamma^{-1}g\gamma \in H}} f(\gamma^{-1}g\gamma).
\]
As $\psi^\ell(f)(g) = f(g^\ell)$, comparing the formula for $\psi^\ell\mathrm{Ind_H^G}(f)$ at $g$ with $\mathrm{Ind_H^G}(\psi^\ell f)$ at $g$, one finds that they differ only in that the former sums over $\gamma^{-1} g^\ell \gamma$ in $H$, whereas the latter sums over $\gamma^{-1} g \gamma$ in $H$. Since the $\ell$th power is a bijection on $H$, as in the proof of \cref{AdamsOpsPermuteIrreps}, the two sums are the same.
\end{proof}

We next consider the endomorphism $\psi^\ell-1$ on $\RU$.
This will later appear as the endomorphism $\psi^\ell-1$ on the nonnegative homotopy Mackey functors of $KU_G$, whose $\Z$-graded 
homotopy Mackey functors are $\RU[\beta^{\pm 1}]$, with $\beta$ in degree 2.
Recall that $\psi^\ell$ acts on $\beta^d$ as multiplication by $\ell^d$ \cite[Proposition 3.2.2]{AtKthy}.

\begin{proposition} \label{psiellinjective}
	Suppose that $\ell$ is coprime to $|G|$.
	The Mackey functor homomorphism $\psi^\ell - 1 \colon \RU\{\beta^d\} \to \RU\{\beta^d\}$ is injective for $d > 0$.
\end{proposition}

\begin{proof}
	This proceeds as the proof of \cite[Proposition 6.8]{BGS}. It suffices to show that this homomorphism is levelwise injective. 
By \cref{AdamsOpsPermuteIrreps}, $\psi^\ell$ acts by permuting the basis of irreducibles in $RU(G)$. If $S$ is the associated permutation matrix, then $\psi^\ell-1$ acts by $\ell^d S - I$, where $I$ is the identity matrix. To show that this matrix is injective as a linear transformation, it suffices to show that it has a nonzero determinant.
	
	 If $d > 0$, this is a matrix with integer entries and $\det(\ell^d S - I) \equiv (-1)^m \pmod \ell$, where $m$ is the number of rows of $S$. 
	 Therefore, $\det(\ell^d S - I) = a \ell + (- 1)^m$ for some $a \in \Z$ (note that $\ell \geq 2$). In particular, it is nonzero. 
\end{proof}

\begin{remark}
\label{psiellnotinjectived0}
The statement of \cref{psiellinjective} in the case $d=0$ does not hold. Indeed, \cite{BGS}*{Proposition~6.7} identifies the kernel of $\psi^\ell-1$ on $RU(G)$ as $\RQ(G)$.
\end{remark}

The result also holds for negative $d$, but there $\ell$ must be invertible in order to define $\psi^\ell$ on $\RU\{\beta^d\}$. We therefore pass to $q$-completion, as this will be the case in which this homomorphism is later considered.

\begin{corollary}
	Suppose that $\ell$ is coprime to $|G|$.
	The Mackey functor homomorphism $\psi^\ell - 1 \colon \RU^\wedge_q\{\beta^d\} \to \RU^\wedge_q\{\beta^d\}$ is injective for $d \neq 0$.
\end{corollary}

\begin{proof}
	In the case that $d$ is positive, this follows from \cref{psiellinjective} by flat base change along $\Z \into \Z^\wedge_q$.
	 For $d < 0$, we argue as in \cref{psiellinjective}. First, $\det(\ell^d S - I) = (\ell^{d})^r \det(S - \ell^{-d} I)$, where $r$ is the number of rows in the matrix. Now $S - \ell^{-d} I$ is an integer matrix with $\det(S - \ell^{-d} I) \equiv \det(S) \pmod \ell$. The permutation matrix $S$ has nonzero determinant, so $\ell^d S - I$ does as well.  
\end{proof}

Having considered the kernel, we now turn to the cokernel.
In order to get a closed form answer, we again
pass to completions, first completing at $q$ in \cref{cokerpsiell} and then completing away from $q$ in \cref{proposition:qnotplocal}.

\begin{notation}
	We will write $\cokpsi\{d\}$ for the cokernel of 
    \(
        \psi^\ell -1\colon \RU\{\beta^d\} \to \RU\{\beta^d\}.
    \)
    We will also write $\cokpsi_p^\wedge\{d\}  = \cokpsi\{d\} \otimes \Z_p^\wedge$ for the cokernel of 
    \(
        \psi^\ell -1\colon \RU_p^\wedge\{\beta^d\} \to \RU_p^\wedge\{\beta^d\},
    \)
    and similarly for the $q$-complete version. When $d = 0$, we sometimes drop the degree from the notation and simply write $\cokpsi$ or $\cokpsi_p^\wedge$. 
\end{notation}

\begin{proposition}
	\label{cokerpsiell}
	Let $\ell$ be primitive mod $|G|=q^j$. The Mackey functor $\cokpsi^\wedge_q\{d\}$ is given at level $G/H$ by:
	\begin{enumerate}[(a)]
		\item for $d \neq 0$, 
	 \[ 
  \cokpsi^\wedge_q\{d\}
		\cong 
		\bigoplus_{\textnormal{cyclic } [C] } \bigsfrac{\Z}{q^{\nu_q(\ell^{d\phi(|C|)} - 1)}} ,
	\]
where the direct sum runs over conjugacy classes of cyclic subgroups $C$ of $H$, $\varphi$ is Euler's totient function, and $\nu_q$ is the $q$-adic valuation. When $|C| = q^k$ with $k \neq 0$, then 
	\[
		\nu_q(\ell^{d \varphi(|C|)} - 1) = {k + \nu_q(d)}.
	\]
		\item for $d = 0$, 
			\[ 
              \cokpsi^\wedge_q
				 \cong \bigoplus_{\textnormal{cyclic } [C]  } \Z_q^\wedge,
			\]
			where the direct sum again runs over conjugacy classes of cyclic subgroups $C$ of $H$.
	\end{enumerate}

\noindent The restriction and transfer in the cokernel are inherited from those in $\RU_q^\wedge$.
\end{proposition}

\begin{proof}
	The cokernel is computed levelwise; at level $G/H$, we have 
	\[ 
		\psi^\ell - 1 \colon RU(H)_q^\wedge\{\beta^d\} \to RU(H)_q^\wedge\{\beta^d\}.
	\] 
	By \cref{AdamsOpsPermuteIrreps}, the Adams operation $\psi^\ell$ permutes the basis of irreducibles of $RU(H)$, and it continues to do so after flat base change along $\Z \to \Z_q^\wedge$. As in the proof of \cref{psiellinjective}, $\psi^\ell - 1$ acts by a matrix $\ell^d S - I$, where $S$ is a permutation matrix and $I$ the identity matrix. Reordering the basis of irreducibles if necessary, this becomes a block-diagonal matrix with blocks 
	\[ 
		\begin{bmatrix}
			-1 & \ell^d \\
			& -1 & \ell^d \\ 
			& & \ddots & \ddots \\ 
			& & & -1 & \ell^d \\
			\ell^d & & & & -1 
		\end{bmatrix}
		\sim
		\begin{bmatrix}
			1 \\
			& 1 \\
			& & \ddots \\
			& & & 1 \\
			& & & & \ell^{dt} - 1
		\end{bmatrix}
	\]
	which are equivalent to diagonal matrices as shown above, using a combination of row and column operations, where $t$ is the number of rows in this block. 
	When $d \neq 0$, each block contributes a factor of $\Z_q^\wedge / (\ell^{dt} - 1)$ to the cokernel. When $d = 0$, each block contributes a factor of $\Z_q^\wedge$. 

	It remains to count the number of blocks and their sizes. Each block corresponds to a $\psi^\ell$-orbit of irreducibles in $RU(H)$. Since $RU(H)$ is a free $\Z$-module of finite rank, we may base change to $\C$ and view the resulting ring as a $\C[\psi^\ell]$-module. Since the character map $\C \otimes RU(H) \to \Cl(H,\C)$ is a map of $\C[\psi^\ell]$-modules and $\C[\psi^\ell]$ is a PID, it suffices to understand the orbits of the $\psi^\ell$ action on a basis for class functions.  The Adams operation acts on class functions by $\psi^\ell(f)(g) = f(g^\ell)$. 
	Consider the basis of $\Cl(H,\C)$ given by the indicator functions $1_{[g]}$.
	Since $\ell$ is primitive mod $|H|$, two indicator functions $1_{[g]}$ and $1_{[h]}$ are in the same $\psi^\ell$-orbit if and only if $g$ and $h$ generate conjugate cyclic subgroups of $H$. Hence, there are as many $\psi^\ell$-orbits as the number of conjugacy classes of cyclic subgroups of $G$.
	 The size of an orbit is the number of generators for a cyclic subgroup; if $C$ is a nontrivial cyclic subgroup of $G$ with $|C| = q^k$, this is $\varphi(q^k) = q^k - q^{k-1} = q^{k-1} (q-1)$. 

Finally, we must understand 
    \[  
       \bigsfrac{\Z_q^\wedge}{(\ell^{d(q-1)q^{k-1}} - 1)}
            \cong 
        \bigsfrac{\Z}{q^{\nu_q(\ell^{d(q-1)q^{k-1}} - 1)}.}
    \] 
 For this, we need to know the largest value $r$ such that $\ell^{d(q-1)q^{k-1}} \equiv 1 \pmod{q^r}$. It helps to work additively. There is an isomorphism of abelian groups $(\Z/q^r)^{\times} \cong \Z/((q-1)q^{r-1})$. Since $\ell$ is a generator of $(\Z/q^r)^{\times}$, it maps to a generator of the right hand side. Since $dq^{k-1}(q-1) \equiv 0 \pmod{(q-1)q^{r-1}}$ when $r \leq k+v_q(d)$, we have  
	\[ 
		\nu_q(\ell^{d(q-1)q^{k-1}} - 1) 
			= k + \nu_q(d). 
	\]
 
So if $C$ is a nontrivial cyclic subgroup of order $q^k$, then the $\psi^\ell$-orbit corresponding to the conjugacy class of $C$
	contributes a factor of
	\[ 
		\bigsfrac{\Z_q^\wedge}{(\ell^{d(q-1)q^{k-1})}-1)} 
		\cong
		\bigsfrac{\Z}{q^{k + \nu_q(d)}}.
	\]
	The trivial cyclic subgroup contributes $\Z_q^\wedge / (\ell^d - 1) \cong \Z / q^{\nu_q(\ell^d - 1)}$.
\end{proof}

\begin{remark}
The formula for the cokernel of $\psi^\ell-1$ in the case $d=0$ holds integrally, before passage to $q$-completion, as can be seen from the proof.
\end{remark}

Levelwise, the formula for $\cokpsi^\wedge_q$ suggests that it is a quotient of $(\uRQ)^\wedge_q$. We show in \cref{ex:cokerpsiell} that this Mackey functor is not a cyclic $\uA^\wedge_q$-module.

\begin{example}
\label{ex:cokerpsiell}
	We calculate $\cokpsi^\wedge_q$ for $G = C_{q^2}$. Recall the representation ring Green functor $\RU_q^\wedge$: 
	\[ 	
		\begin{tikzcd}[row sep=50]
			\Z_q^\wedge[x]/(x^{q^2} - 1)
				\ar[d, bend right=50,
					"{
						x \mapsto y
					}"{left}
					]
				\\
			\Z_q^\wedge[y]/(y^q - 1)
				\ar[u, bend right=50,
					"{	\displaystyle
						y^k \mapsto \sum_{i=0}^{q-1} x^{iq+k}
					}"{right}
					]
				\ar[d, bend right=50, 
					"{
						y \mapsto 1
					}"{left} 
					]
				\\
			\Z_q^\wedge
				\ar[u, bend right=50, 
					"{	\displaystyle
						1 \mapsto \sum_{i=0}^{q-1} y^i
					}"{right}
					]
		\end{tikzcd}
	\]
	Here, $y$ is the class of the $C_q$-representation where the generator acts on the complex plane by a $q$-th root of unity, and $x$ is the class of the $C_{q^2}$-representation where the generator acts by a primitive $q^2$ root of unity. Since these are one-dimensional complex representations, the Adams operation takes $y$ to $y^\ell$ and $x$ to $x^\ell$. Hence, the Mackey functor homomorphism $\psi^\ell - 1$ takes the form
	\[ 	
		\begin{tikzcd}[row sep=50, column sep=100]
			\Z_q^\wedge[x]/(x^{q^2} - 1)
				\ar[d, bend right=50,
					"{
						\res_{C_q}^{C_{q^2}}
					}"{left}
					]
				\ar[r, "\displaystyle x^k \mapsto x^{k\ell} - x^k"]
				& 
			\Z_q^\wedge[x]/(x^{q^2} - 1)
				\ar[d, bend right=50,
					"{
						\res_{C_q}^{C_{q^2}}
					}"{left}
					]
				\\
			\Z_q^\wedge[y]/(y^q - 1)
				\ar[u, bend right=50,
					"{	
						\tr_{C_q}^{C_{q^2}}
					}"{right}
					]
				\ar[d, bend right=50, 
					"{
						\res_e^{C_q}
					}"{left} 
					]
				\ar[r, "\displaystyle y^k \mapsto y^{k\ell} - y^k"]
				&
			\Z_q^\wedge[y]/(y^q - 1)
				\ar[u, bend right=50,
					"{	
						\tr_{C_q}^{C_{q^2}}
					}"{right}
					]
				\ar[d, bend right=50, 
					"{
						\res_e^{C_q}
					}"{left} 
					]
				\\
			\Z_q^\wedge
				\ar[u, bend right=50, 
					"{	
						\tr_e^{C_q}
					}"{right}
					]
				\ar[r, "\displaystyle 0"]
				&
			\Z_q^\wedge
				\ar[u, bend right=50, 
					"{	
						\tr_{C_q}^{C_{q^2}}
					}"{right}
					]
				\\[-1cm]
			\RU_q^\wedge
				\ar[r, "\psi^\ell - 1"]
				& 
			\RU_q^\wedge 
		\end{tikzcd}
	\]

	At the $C_q$-level, the quotient by the image of $\psi^\ell -1$ identifies all nontrivial representations because $\ell$ is primitive mod $q$. At the top level, the quotient places the nontrivial representations into two classes: the class of $x$ and the class of $x^q$. Hence, the cokernel is:
	\[ 
		\begin{tikzcd}[row sep=huge,ampersand replacement=\&]
			\Z_q^\wedge\{1, x, x^q\} 
				\ar[d, 
					bend right=30,
					"{\begin{pmatrix}
						1 & 0 & 1 \\ 0 & 1 & 0
					\end{pmatrix}}"{left}
					bend right=50,
					"{
						\begin{pmatrix}
							1 & 0 & 1 \\ 0 & 1 & 0
						\end{pmatrix}
					}"{left}
					]
				\\
			\Z_q^\wedge\{1, y\}
				\ar[u, 
					bend right=50,
					"{
						\begin{pmatrix}
							1 & 0 \\ 
							0 & q \\ 
							q-1 & 0
						\end{pmatrix}
					}"{right}
					]
				\ar[d, bend right=50, 
					"{
						\begin{pmatrix}
							1 & 1
						\end{pmatrix}
					}"{left} 
					]
				\\
			\Z_q^\wedge
				\ar[u, bend right=50, 
					"{
						\begin{pmatrix}
							1 \\ q-1
						\end{pmatrix}
					}"{right} 
					]
				\\[-1cm]
			\cokpsi_3^\wedge
		\end{tikzcd}
	\]	
	This Mackey functor is not free; it contains a copy of the Burnside Mackey functor generated by the element $1$ at the top level, and the quotient by this subfunctor has $q$-torsion. 
\end{example}

\begin{example}
\label{ex:cokerpsiC9}
Let $G = C_9$. Below we present the Mackey functors $\cokpsi_3^\wedge\{1\}$ and $\cokpsi_3^\wedge\{2\}$, which are the cokernels of $\psi^\ell - 1$ on $\RU_3^\wedge\{\beta\}$ and $\RU_3^\wedge\{\beta^2\}$, respectively. 
	\[ 
		\begin{tikzcd}[row sep=huge,ampersand replacement=\&]
			\Z/3\{x^3\} \oplus \Z/{9}\{x\}
				\ar[d, 
					bend right=40,
					"{\begin{pmatrix}
						0 & 1 
					\end{pmatrix}}"{left}
					]
				\\
			\Z/3\{ y\}
				\ar[u, 
					bend right=40,
					"{
						\begin{pmatrix}
							0 \\ 
							3 \\ 
						\end{pmatrix}
					}"{right}
					]
				\ar[d, bend right=40, ]
				\\
			0
				\ar[u, bend right=40, ]
				\\[-1cm]
				\cokpsi^\wedge_3\{1\}
		\end{tikzcd}
		\qquad \qquad
		\begin{tikzcd}[row sep=huge,ampersand replacement=\&]
			\Z/3\{1,x^3\} \oplus \Z/{9}\{x\}
				\ar[d, 
					bend right=40,
					"{\begin{pmatrix}
						1 & 1 & 0 \\
						0 & 0 & 1
					\end{pmatrix}}"{left}
					]
				\\
			\Z/3\{1, y\}
				\ar[u, 
					bend right=40,
					"{
						\begin{pmatrix}
							1 & 0 \\ 
							2 & 0 \\
							0 & 3 \\ 
						\end{pmatrix}
					}"{right}
					]
				\ar[d, bend right=40, 
					"{\begin{pmatrix}
						1 & 1 
					\end{pmatrix}}"{left}
				]
				\\
			\Z/3
				\ar[u, bend right=40, 
					"{\begin{pmatrix}
						1 \\ 2 
					\end{pmatrix}}"{right}
				]
				\\[-1cm]
				\cokpsi^\wedge_3\{2\}
		\end{tikzcd}
	\]	
\end{example}

In the case of completing at $p$ different from $q$, it turns out that the cokernel of $\psi^\ell-1$ has a familiar form.

\begin{proposition}
\label{proposition:qnotplocal}
	Let $\ell$ be primitive mod $|G|=q^j$ and let $p$ be a prime different from $q$. There is an isomorphism of $G$-Mackey functors
	\[ 
            \cokpsi^\wedge_p
			\cong 
			(\uRQ)_p^\wedge.
	\]
\end{proposition}

To prove this proposition, we make use of explicit formulas for the equivalence of categories \cite[Proposition 7.6]{BGS} (originally due to \cite[Theorem A.9 and Proposition A.12]{GrMa}): 
\begin{equation}
\label{PLocalMackeySplits}
	\Mack(G)_{(p)} \xrightarrow[\simeq]{(V_H)} \bigoplus_{(H)} \Mod_{\Z_{(p)}[W_G(H)]}.
\end{equation}
Here, $\Mack(G)_{(p)}$ is the localization of the category of $G$-Mackey functors  at $p$, i.e. Mackey functors valued in $\Z_{(p)}$-modules rather than abelian groups. The sum on the right hand side of this equivalence runs over all conjugacy classes of subgroups of $G$. The functor $V_H$ sends $\uM \in \Mack(G)_{(p)}$ to the quotient of $\uM(G/H)$ by the subgroup generated by transfers from all proper subgroups of $H$ \cite[Proposition 7.10]{BGS}.

\begin{proof}[Proof of \cref{proposition:qnotplocal}]
	Under the equivalence of categories \cref{PLocalMackeySplits}, $(\uRQ)_p^\wedge$ maps to 
	\[ 
		V_H\left((\uRQ)_p^\wedge\right) = 
		\begin{cases}
			\Z_p^\wedge & H \text{ cyclic,}\\
			0 & \text{ otherwise}
		\end{cases}
	\]
	by \cite[Proposition 7.11]{BGS}. It suffices to show that $\cokpsi^\wedge_p$ has isomorphic image under this equivalence. 

If $H = C_{q^k}$ is a cyclic subgroup of $G$, $RU(C_{q^k})_p^\wedge \cong \Z_p^\wedge[x]/(x^{q^k} - 1)$. The ideal of transfers from proper subgroups is generated by transfers from $C_{q^{k-1}}$; each such transfer is a multiple of the cyclotomic polynomial $\phi_{q^k}(x) = \tr_{C_{q^{k-1}}}^{C_{q^k}}(1)$. Therefore, 
\[ 
	V_{C_{p^k}}(\RU_p^\wedge) \cong \Z_p^\wedge[x]/ \phi_{q^k}(x).
\]
See also \cite{TW95}*{Example~6.7}.

If $H$ is not cyclic, there exists a surjection $\theta\colon H \to C_q \times C_q$. 
In \cref{Lemma:qInTransferIdeal}, we will show that $q\in RU(C_q\times C_q)$ lies in the image of transfers from proper subgroups.
A double coset formula yields the commuting diagram
\[
	\begin{tikzcd}
	\displaystyle\bigoplus\limits_{C \lneqq C_q \times C_q }\hspace{-1em} RU(C)
	\ar[r,"\theta^*"]
	\ar[d,"\mathrm{tr}"]
	&
	\displaystyle\bigoplus\limits_{C \lneqq C_q \times C_q }\hspace{-1em} RU(\theta^{-1} C)
	\ar[d,"\mathrm{tr}"]
	\\
	RU(C_q\times C_q)
	\ar[r,"\theta^*"]
	&
	RU(H).
\end{tikzcd}
\]
Since $\theta^*$ is a ring homomorphism, $\theta^*(q) = q$. 
This shows that $q\in RU(H)$ lies in the image of transfers from proper subgroups.
Since $q$ becomes a unit after $p$-completion, $V_H(\RU_p^\wedge)$ is the quotient of $\RU_p^\wedge(H)$ by the unit ideal and therefore vanishes.
The rational version of this statement appears, for example, in \cite{Thev}*{Section~9}.

We have seen that
	\[ 
		V_H(\RU_p^\wedge) = 
		\begin{cases}
			\Z_p^\wedge[x]/\phi_{q^k}(x) & H \text{ cyclic and } |H| = q^k,\\
			0 & \text{ otherwise},
		\end{cases}
	\]
	where $\phi_{q^k}(x)$ is the $q^k$-th cyclotomic polynomial. Hence, it remains to determine the cokernel of 
\begin{equation}
	\label{equation:AdamsOpOnZqAdjoinZeta}
		\Z_p^\wedge[x]/\phi_{q^k}(x) 
			\xrightarrow{\psi^\ell - 1} 
		\Z_p^\wedge[x]/\phi_{q^k}(x)
	\end{equation}

	We may write $\Z_p^\wedge[x]/\phi_{q^k}(x) \cong \Z_p^\wedge\{ x,x^2,x^3,\dots, x^{(q-1)q^{k-1}}\}$.
	The Adams operation $\psi^\ell$ cyclically permutes the $q-1$ powers of $x^{q^{k-1}}$ in this basis.
	Thus we may decompose
	\[
		\Z_p^\wedge[x]/\phi_{q^k}(x) \cong  A \oplus B
	\]
	as a $\Z_p^\wedge[\psi^\ell]$-module, where
	\[
		A = \Z_p^\wedge\{ x^{i q^{k-1}}  \mid  1 \leq i \leq q-1 \} 
	\]
	and
	\[
		B =  \Z_p^\wedge\{ x^n \mid q^{k-1} \text{ does not divide } n, 1 \leq n < (q-1)q^{k-1} \}.
	\]
	It then follows that the cokernel of $\psi^\ell-1$ on $A$ gives $\Z^\wedge_p$.
	We claim that  on $B$ the cokernel vanishes. 
	
	Primitivity of $\ell$ ensures that in $RU(C_{q^k}) \cong \Z[x]/(x^{q^k}-1)$, two monomials $x^{n_1}$ and $x^{n_2}$ are in the same $\psi^\ell$-orbit if and only if $n_1$ and $n_2$ have the same $q$-adic valuation, where $n_1$ and $n_2$ are both assumed to be less than $q^k$. It then follows that in the cokernel of $\psi^\ell-1$ on $RU(C_{q^k})$, the polynomial $\phi_{q^k}(x)\cdot x^n$ is equivalent to $q\cdot x^n$, so long as $n$ is not divisible by $q^{k-1}$. Thus in the cokernel of $\psi^\ell-1$ on the quotient ring $\Z[x]/\phi_{q^k}(x)$, there is a relation $q\cdot x^n=0$ when $n$ is not divisible by $q^{k-1}$.
 In particular, after completing at $p$, which is different from $q$, it follows that $x^n$ vanishes in the cokernel of $\psi^\ell-1$ on $\Z[x]/\phi_{q^k}(x)$.

	Hence, $\cokpsi^\wedge_p$ corresponds to $\Z_p^\wedge$ supported on the cyclic subgroups, and under the equivalence of categories, this is the same as $(\uRQ)_p^\wedge$. Since the equivalence \cref{PLocalMackeySplits} preserves and creates cokernels, we are done. 
\end{proof}

\begin{lemma}
\label{Lemma:qInTransferIdeal}
	The ideal of $RU(C_q \times C_q)$ generated by transfers from proper subgroups contains $q$. 
\end{lemma}

\begin{proof}
	Let $H = C_q \times C_q$. The representation ring of $H$ is isomorphic as a commutative ring to $\Z[x,y] / (x^q - 1, y^q - 1)$, where $x$ and $y$ are the classes of rotation representations of the left and right factors, respectively. If $K$ is the subgroup of $H$ generated by an element $(\gamma^i,\gamma^j)$ with $i \neq 0$, then 
	\[ 
		\tr_K^{H}(1) = \sum_{k = 0}^{q-1} (x^i y^j)^k.
	\]
	We claim that 
	\[ 
		q = \sum_{K \leq H} \tr_{K}^{H}(1)  - \tr_{L}^{H}(1) \cdot \tr_{R}^{H}(1),
	\]
	where $L$ is the subgroup generated by $(\gamma,e)$ and $R$ is the subgroup generated by $(e,\gamma)$. This is a calculation. Recall that $H$ has $q+1$ distinct subgroups of order $q$: the subgroup $R$ generated by $(e,\gamma)$, and subgroups generated by elements $(\gamma,\gamma^j)$ for $j = 0,1,\ldots,q-1$. 
	
	\begin{align*}
		\sum_{K \leq H} \tr_{K}^{H}(1)& - \tr_L^H(1) \cdot \tr_R^H(1) 
		\\
		& = \sum_{j = 0}^{q-1} \tr_{\langle (\gamma,\gamma^j) \rangle}^H (1) + \tr_{R}^H(1) - \tr_L^H(1) \cdot \tr_R^H(1) \\
		& = \sum_{j = 0}^{q-1} \tr_{\langle (\gamma,\gamma^j) \rangle}^H (1) + \tr_R^H(1) \cdot (1 - \tr_L^H(1))\\
		&= \sum_{j=0}^{q-1} \left(1 + xy^j + x^2 y^{2j} + \ldots + x^{q-1} y^{j(q-1)} \right) +\\
		& \hspace*{1cm} \left(1 + y + y^2 + \ldots + y^{q-1}\right) \left( - x - x^2 - \ldots - x^{q-1} \right) \\
		&= q + \sum_{j=0}^{q-1} \left(xy^j + x^2 y^{2j} + \ldots + x^{q-1} y^{j(q-1)} \right) \\
		& \qquad - \sum_{k=0}^{q-1} \left( xy^k + x^2y^k + \ldots + x^{q-1}y^k \right)\\
		& = q.
	\end{align*}
	The last equality follows by a reindexing, recalling that these equations live in the ring $\Z[x,y] / (x^q - 1, y^q - 1)$. 
\end{proof}

\begin{example}
	Let $H = C_3 \times C_3$, and let $L = \langle (\gamma,e) \rangle$, $C = \langle (\gamma,\gamma) \rangle$, $D = \langle (\gamma,\gamma^2) \rangle$, and $R = \langle (e,\gamma) \rangle$ be its four subgroups of order $3$. Consider the representation ring Green functor for $C_3 \times C_3$. In this Green functor, we have: 
	\begin{align*}
		\tr_L^H(1) &= 1 + x + x^2 \\
		\tr_C^H(1) &= 1 + xy + x^2 y^2 \\
		\tr_D^H(1) &= 1 + x^2 y + xy^2 \\
		\tr_R^H(1) &= 1 + y + y^2
	\end{align*}

\noindent We can directly see that $3$ is contained in the ideal of $RU(H)$ generated by images of transfers from proper subgroups of $H$: 
	\begin{align*}
		\sum_{K \leq H} &\tr_K^{H}(1) - \tr_L^{H}(1) \cdot \tr_R^{H}(1) \\
			&= \big(\tr_L^H(1) + \tr_R^H(1) + \tr_C^H(1) + \tr_D^H(1)\big) - \tr_L^H(1) \cdot \tr_R^H(1)\\
			&= \left(1 + x + x^2 \right)
				+ \left(1 + y + y^2\right)
				+ \left(1 + xy + x^2 y^2 \right)
				+ \left(1 + x^2y + y^2x\right) \\
			& \qquad \qquad	- \left(1 + x + x^2\right)\left(1 + y + y^2\right)\\
			&= \left(4 + x + y + x^2 + xy + y^2 + x^2y + y^2x + x^2y^2\right) \\
			& \qquad 	- \left(1 + x + x^2 + y + xy x^2y + y^2 + y^2x + y^2 x^2\right)\\
			& = 3.
	\end{align*}
\end{example}

\section{The homotopy Mackey functors of $L_{\KUG}\S_G$}
\label{sec:homotopy}

Our strategy for understanding $L_{\KUG}\S_G$ is to use the fracture square \eqref{fracturesquare}. We begin by describing the homotopy Mackey functors of the local factors $L_{\KUG/p}\S_G$, both in the case $p=q$ and $p \neq q$, using the work of \cref{sec:algebra}.
With the local computations in hand, we then use the long exact sequence
\begin{equation}
\label{eq:LES}
	\cdots \to \Q\otimes\prod\limits_p\upi_{n+1} L_{\KUG/p}\S_G \to \upi_n L_{\KUG} \S_G \to \upi_n \Q \otimes L_\KUG \S_G \times \prod\limits_p \upi_n L_{\KUG/p} \S_G \to \dots
\end{equation}
arising from the fracture square \eqref{fracturesquare} to obtain the homotopy Mackey functors  $\upi_n  L_\KUG \S_G$.
We will use the fact that the rationalization $\Q\otimes L_\KUG \S_G$ is the Eilenberg-Mac~Lane spectrum for the rational Mackey functor $\Q\otimes \uRQ$ \cite{BGS}*{Lemma~9.1}.

\subsection{Local computations for $p = q$}
\label{sec:localpisq}

In this section, we compute the homotopy Mackey functors for $L_{KU_G/q}\S_G$.
The key tool is the following.

\begin{proposition}[{\cite[Propositions 5.3, 6.3]{BGS}}]
\label{fiberseq}
	If $\ell$ is primitive modulo $|G|$, then the Adams operation $\psi^\ell \colon (\KUG)_q^\wedge \to (\KUG)_q^\wedge$ is a well-defined map of $G$-spectra that participates in a fiber sequence
	\begin{equation} 
	\label{eq:fiberseq}
    	L_{\KU_G/q}\S_G 
    		\longrightarrow 
    	(\KU_G)_q^\wedge 
    		\xrightarrow{\psi^\ell - 1} 
    	(\KU_G)_q^\wedge.
	\end{equation}
\end{proposition}

Note that the fiber is independent of $\ell$ in the fiber sequence above. 

\begin{remark}
	The original proposition 5.3 in \cite{BGS} contains the assumption that $\ell$ is primitive mod $|G|$, but by \cite[Corollary 2.5]{HiKo}, in order to show that $\psi^\ell$ extends to a map of $G$-spectra, it suffices to assume that $\ell$ is coprime to $|G| = q^k$.
	On the other hand, in order to identify the fiber as the $KU_G/q$-local equivariant sphere, the additional primitivity assumption is required.
\end{remark}

Since $\upi_*\KUG \cong \RU[\beta, \beta^{-1}]$ with {$\beta$ in degree $2$}, the long exact sequence of homotopy Mackey functors associated to the fiber sequence \cref{eq:fiberseq} splits into four-term exact sequences: 
\[ 
	\begin{tikzcd}
		0 
			\ar[r] 
			&
		\upi_{2d} L_{\KUG/q} \S_G
			\ar[r]
			&
		\RU_q^\wedge\{\beta^d\}
			\ar[r, "\psi^\ell - 1"] 
			& 
		\RU_q^\wedge\{\beta^d\}
			\ar[r]
			&
		\upi_{2d-1} L_{\KUG/q} \S_G
			\ar[r]
			& 
		0
	\end{tikzcd}
\]
{Thus the homotopy Mackey functors of $L_{\KUG/q} \S_G$ follow from the work of \cref{sec:algebra}.}

\cref{psiellinjective} immediately implies the following.

\begin{corollary}
\label{cor:pi2dvanishes}
	For $d \neq 0$, $\upi_{2d} L_{\KUG/q} \S_G = 0$.
\end{corollary}

The $d=0$ case was previously computed:

\begin{proposition} \cite[Proposition 6.8]{BGS}
$\upi_0 L_{\KUG/q}\S_G \cong (\uRQ)^\wedge_q$.
\end{proposition}

\begin{corollary}
\label{cor:pi2dmin1coker}
	The  Mackey functor $\upi_{2d-1} L_{\KUG/q} \S_G$ is
	\[
		\upi_{2d-1} L_{\KUG/q} \S_G \cong \cokpsi^\wedge_q\{d\} = 
		\coker\Big(\RU^\wedge_q\{\beta^d\} \xrightarrow{\psi^\ell-1} \RU^\wedge_q\{\beta^d\} \Big).
	\]
\end{corollary}

This cokernel was computed in \cref{cokerpsiell}.

\begin{example}
The Mackey functors $\cokpsi^\wedge_3\{1\}$ and $\cokpsi^\wedge_3\{2\}$ were computed for $G=C_9$ in 
\cref{ex:cokerpsiC9}.
According to \cref{cor:pi2dmin1coker}, these agree with the homotopy Mackey functors $\upi_1 L_{KU_{C_{9}}/3} \S_{C_{9}}$ and $\upi_3 L_{KU_{C_{9}}/3} \S_{C_{9}}$.
\end{example}

\subsection{Local computations for $p \neq q$}
\label{sec:localpnotq}

Let $p$ be a prime that does not divide $|G| = q^k$. The calculation of the nonzero $p$-local homotopy groups of $L_{\KUG}\S_G$ was done in \cite{BGS}. For an odd prime $p$, recall the homotopy groups of $L_{\KU/p} \S$ 
{as originally calculated by Adams--Baird \cite{AB} and Ravenel \cite{Rav} and described more recently in }
\cite[Equation 2.3.8]{zhang}:
\[ 
	\pi_n L_{\KU\!/p} \S 
		\cong
		\begin{cases}
			\Z_p^\wedge 
				& 
			\text{if } n \in \{0, -1\},
				\\
			\Z/p^{\nu_p(k) + 1} 
				& 
			\text{if } n = 2k-1 \text{ and } (p-1) \mid k,
				\\
			0
				& 
			\text{otherwise.}
		\end{cases}
\]

For $p = 2$, the homotopy groups of $L_{\KU/2} \S$ are in \cite[Equation 2.3.13]{zhang}:
\[
	\pi_n L_{\KU\!/2} \S 
		\cong
		\begin{cases}
			\Z_2^\wedge \oplus \Z / 2 
				& 
			\text{if } n = 0,
				\\
			\Z_2^\wedge 
				& 
			\text{if } n = -1,
				\\
			\Z/2 \oplus \Z/2 
				& 
			\text{if } n \equiv 1 \pmod 8,
				\\
			\Z/2 
				& 
			\text{if } n \equiv 0,2 \pmod 8, \text{ and } n \neq 0, 
				\\
			\Z/2^{\nu_2(k) + 3} 
				& 
			\text{if } n = 4k - 1 \text{ and } n \neq -1,
				\\
			0 
				& 
			\text{otherwise.}
		\end{cases}
\]

\begin{proposition}[{\cite[Proposition 8.5]{BGS}}]
\label{prop:htpyLKUGmodp}
    Let $p\neq q$. 
	There is an isomorphism of graded Green functors 
	\[ 
		\upi_* L_{\KUG/p} \S_G \cong \uRQ \otimes \pi_* L_{\KU/p}\S. 
	\]
\end{proposition}

The above is a complete description of the $p$-complete homotopy Mackey functors of $L_{\KUG} \S_G$, but 
\cref{proposition:qnotplocal} then gives the following description in the case $n=-1$:

\begin{corollary}
\label{cor:piminus1pnotq}
For $p\neq q$, we have
$ 
	\upi_{-1} L_{\KUG/p} \S_G 
		\cong 
	\uRQ \otimes \Z_p^\wedge 
		\cong \cokpsi^\wedge_p.
$
\end{corollary}

\subsection{Local to global reassembly}
\label{sec:Reassembly}

Here we use the work of \cref{sec:localpisq} and \cref{sec:localpnotq}, in combination with the long exact sequence \eqref{eq:LES}, to deduce the homotopy Mackey functors  $\upi_n L_\KUG \S_G$.
The case $n=0$ was the focus of \cite{BGS}. 
The cases $n=-1$ and $n=-2$ behave quite differently from the rest, so we begin by considering the cases  of $n$ different from $0$, $-1$, or $-2$.

\begin{proposition} 
\label{upi2k}
Let $n=2k$ be different from $0$ and $-2$. Then 
\[\upi_{2k} L_\KUG \S_G \cong 
\uRQ \otimes \pi_{2k}L_{KU}\S \cong \uRQ \otimes \Z/2\]
for  $2k\equiv 0,2 \pmod8$. This Mackey functor vanishes otherwise.
\end{proposition}

\begin{proof}
    Fix $2k$ different from $0$ and $-2$.
    By \cref{cor:pi2dvanishes} and \cref{prop:htpyLKUGmodp}, we have that for 
$p$ any odd prime (including $p=q$), then $\upi_{2k} \left( L_{\KUG/p} \S_G \right)$ vanishes. In the case of $p=2$, we have
\[ 
	\upi_{2k} \left( L_{\KUG/2} \S_G \right) \cong 
	\begin{cases} 
		\uRQ \otimes \Z/2 & 2k \equiv 0,2 \pmod8, \\
		0 & \text{else}.
	\end{cases}
\]
Similarly, we find that $\upi_{2k+1} L_{\KUG/p} \S_G$ is nonzero (and levelwise finite) only for finitely many primes $p$. It follows that $ \Q \otimes \prod_p \upi_{2k+1} L_{\KUG/p}\S_G$ vanishes. The result now follows from \eqref{eq:LES}.
\end{proof}

In the case of $n$ odd and different from $-1$, the answer is stated in terms of the cokernel of $\psi^\ell-1$, where as usual $\ell$ is primitive modulo the order of $G$.

\begin{proposition}
\label{upi2kminus1}
Let $2k-1 \neq -1$. Then 
\[
	\upi_{2k-1} L_\KUG \S_G \cong \uRQ \otimes \pi_{2k-1} L_{KU}\S\big[ \tfrac1q\big] \oplus \cokpsi^\wedge_q\{k\}.
\]
\end{proposition}

\begin{proof}
According to \cref{sec:localpnotq}, the homotopy Mackey functors of $L_{\KUG/p}\S_G$ are levelwise finite in degrees $2k$ and $2k-1$. 
\cref{cor:pi2dvanishes} gives that $\upi_{2k}L_{\KUG/q}\S_G$ vanishes, while \cref{cor:pi2dmin1coker} identifies $\upi_{2k-1}L_{\KUG/q}\S_G$ with $\cokpsi^\wedge_q\{d\}$. By \cref{cokerpsiell}, this is levelwise finite.
\end{proof}

We now turn our attention to the case $n=-1$. 

\begin{proposition} 
 $\upi_{-1} L_\KUG \S_G=0$.
\end{proposition}

\begin{proof}
By \cref{cokerpsiell}{(b)} and \cref{cor:pi2dmin1coker}, the Mackey functor $\upi_{-1}L_{\KUG/q} \S_G$ is torsion-free. The same is true of $\upi_{-1} L_{\KUG/p}\S_G$ for $p\neq q$ by \cref{sec:localpnotq}. It follows that the map
\[
	\prod\limits_p \upi_{-1} L_{\KUG/p} \S_G \longrightarrow \Q\otimes \left( \prod\limits_p \upi_{-1} L_{\KUG/p} \S_G \right)
\]
is injective. 
The long exact sequence \eqref{eq:LES} then shows that $\upi_{-1}L_\KUG \S_G$ is the cokernel of 
\[
	\Q\otimes \uRQ \times \prod\limits_p \upi_0 L_{\KUG/p} \S_G \longrightarrow \Q\otimes \left( \prod\limits_p \upi_0 L_{\KUG/p} \S_G \right), 
\]
which may be rewritten as
\[
	\Q\otimes \uRQ \oplus  \Z/2 \otimes \uRQ \oplus \prod\limits_p (\uRQ)^\wedge_p \longrightarrow \Q\otimes \left( \prod\limits_p (\uRQ)^\wedge_p \right).
\]
It suffices to show that this is levelwise surjective. As the values of the Mackey functor $\uRQ$ are all free abelian groups of finite rank, the result follows from \cref{lem:piminus1}.
\end{proof}

\begin{lemma}
\label{lem:piminus1}
Let $B$ be a free abelian group of finite rank. Then the map
\[
	f\colon (\Q \otimes B) \oplus \prod\limits_p B^\wedge_p \longrightarrow \Q \otimes \left( \prod\limits_p B^\wedge_p \right) 
\]
defined by
\[
	f\left( \frac{b_0}{n}, (b_p) \right) = \frac1n (b_0-nb_p)
\]
is surjective. 
\end{lemma}

\begin{proof}
Left to the reader.
\end{proof}

Finally, we deal with the case $n=-2$.

\begin{proposition}
\label{prop:piminus2}
$ \upi_{-2}L_\KUG \S_G \cong \Q/\Z \otimes \cokpsi$.
\end{proposition}

\begin{proof}
By \cref{cor:pi2dvanishes} and \cref{sec:localpnotq}, the Mackey functors $\upi_{-2}L_{\KUG/p}\S_G$ vanish for all primes $p$. It follows from the long exact sequence \eqref{eq:LES} that $\upi_{-2}L_\KUG \S_G$ is the cokernel of the rationalization map
\[
	\prod\limits_p \upi_{-1} L_{\KUG/p} \S_G \longrightarrow \Q\otimes \left( \prod\limits_p \upi_{-1} L_{\KUG/p} \S_G \right).
\]
In other words, we have that 
\[
	\upi_{-2}L_\KUG \S_G \cong \Q/\Z \otimes \left( \prod\limits_p \upi_{-1} L_{\KUG/p} \S_G \right).
\]
By \cref{cor:pi2dmin1coker} and \cref{cor:piminus1pnotq}, this may be rewritten as 
\[
	\upi_{-2}L_\KUG \S_G \cong \Q/\Z \otimes \left( \prod\limits_p \cokpsi^\wedge_p \right).
\]
Each Mackey functor $\cokpsi^\wedge_p$ is (levelwise) $p$-local, so that according to \cref{lemma:producttosumtorsion} we have an isomorphism 
\[\begin{split}
	\Q/\Z \otimes \left( \prod\limits_p \cokpsi^\wedge_p \right) &\cong
	\bigoplus_p \left( \Q_p/\Z_p \otimes \cokpsi^\wedge_p \right)  \cong
	\bigoplus_p \left( \Q_p/\Z_p \otimes \cokpsi \right) \\
	& \cong \left( \bigoplus_p \Q_p/\Z_p \right) \otimes \cokpsi \cong \Q/\Z \otimes \cokpsi.
\end{split}\]
\end{proof}

\begin{lemma}
\label{lemma:producttosumtorsion}
	Suppose for each prime $p$, $A_p$ is an abelian group such that all primes different than $p$ act invertibly on $A_p$. Then 
	\[ 
		\Q/\Z  \otimes \left(\prod_p A_p \right) 
			\cong 
		\bigoplus_p \left(\Q_p/ \Z_p \otimes  A_p\right). 
	\]
\end{lemma}

\begin{proof}
This follows from the decomposition of $\Q/\Z$ as $\bigoplus_r \Q_r/\Z_r$ as $r$ runs over primes, the expression of $\Q_r/\Z_r$ as $\colim_k \Z/r^k$, and the fact that tensor product commutes with colimits.
\end{proof}


\begin{bibdiv}
\begin{biblist}

\bib{AB}{incollection}{
	label = {Ad}
	title = {Operations of the $n$-th kind in $K$-theory, and what we don't know about $\mathbb{RP}^\infty$}
	booktitle = {New Developments in Topology},
	author = {Adams, J. F.},
	editor = {Segal, Graeme},
	year = {1974},
	pages = {1-10},
	publisher = {Cambridge University Press},
	DOI={10.1017/CBO9780511662607.002}
}


\bib{AtKthy}{book}{
   author={Atiyah, M. F.},
   title={$K$-theory},
   series={Advanced Book Classics},
   edition={2},
   note={Notes by D. W. Anderson},
   publisher={Addison-Wesley Publishing Company, Advanced Book Program,
   Redwood City, CA},
   date={1989},
   pages={xx+216},
   isbn={0-201-09394-4},
   review={\MR{1043170}},
}
	


\bib{BGS}{article}{
  author={Bonventre, Peter J.},
  author={Guillou, Bertrand J.},
  author={Stapleton, Nathaniel J.},
  title={On the $KU_G$-local equivariant sphere},
  eprint={https://arxiv.org/abs/2204.03797},
}


\bib{Bousfield}{article}{
    label = {Bo},
   author={Bousfield, A. K.},
   title={The localization of spectra with respect to homology},
   journal={Topology},
   volume={18},
   date={1979},
   number={4},
   pages={257--281},
   issn={0040-9383},
   review={\MR{551009}},
   doi={10.1016/0040-9383(79)90018-1},
}

\bib{TMF}{collection}{
   title={Topological modular forms},
   series={Mathematical Surveys and Monographs},
   volume={201},
   editor={Douglas, Christopher L.},
   editor={Francis, John},
   editor={Henriques, Andr\'{e} G.},
   editor={Hill, Michael A.},
   publisher={American Mathematical Society, Providence, RI},
   date={2014},
   pages={xxxii+318},
   isbn={978-1-4704-1884-7},
   review={\MR{3223024}},
   doi={10.1090/surv/201},
}

\bib{GrMa}{article}{
   author={Greenlees, J. P. C.},
   author={May, J. P.},
   title={Generalized Tate cohomology},
   journal={Mem. Amer. Math. Soc.},
   volume={113},
   date={1995},
   number={543},
   pages={viii+178},
   issn={0065-9266},
   review={\MR{1230773}},
   doi={10.1090/memo/0543},
}




\bib{HiKo}{article}{
   author={Hirata, Koichi},
   author={Kono, Akira},
   label={HiKo},
   title={On the Bott cannibalistic classes},
   journal={Publ. Res. Inst. Math. Sci.},
   volume={18},
   date={1982},
   number={3},
   pages={1187--1191},
   issn={0034-5318},
   review={\MR{688953}},
   doi={10.2977/prims/1195183304},
}
		

\bib{Narkiewicz}{book}{
   author={Narkiewicz, W\l adys\l aw},
   title={The development of prime number theory},
   series={Springer Monographs in Mathematics},
   note={From Euclid to Hardy and Littlewood},
   publisher={Springer-Verlag, Berlin},
   date={2000},
   pages={xii+448},
   isbn={3-540-66289-8},
   review={\MR{1756780}},
   doi={10.1007/978-3-662-13157-2},
}
	
		
\bib{Rav}{article}{
   label = {Ra},
   author={Ravenel, Douglas C.},
   title={Localization with respect to certain periodic homology theories},
   journal={Amer. J. Math.},
   volume={106},
   date={1984},
   number={2},
   pages={351--414},
   issn={0002-9327},
   review={\MR{737778}},
   doi={10.2307/2374308},
}
		


\bib{Ritter}{article}{
   label = {Ri},
   author={Ritter, J\"{u}rgen},
   title={Ein Induktionssatz f\"{u}r rationale Charaktere von nilpotenten
   Gruppen},
   language={German},
   journal={J. Reine Angew. Math.},
   volume={254},
   date={1972},
   pages={133--151},
   issn={0075-4102},
   review={\MR{470058}},
   doi={10.1515/crll.1972.254.133},
}


\bib{Seg-perm}{article}{
  label = {Seg2},
  author={Segal, Graeme},
  title={Permutation representations of finite $p$-groups},
  journal={Quart. J. Math. Oxford Ser. (2)},
  volume={23},
  date={1972},
  pages={375--381},
  issn={0033-5606},
  review={\MR{322041}},
  doi={10.1093/qmath/23.4.375},
}


\bib{Serre}{book}{
   label = {Ser},
   author={Serre, Jean-Pierre},
   title={Linear representations of finite groups},
   series={Graduate Texts in Mathematics, Vol. 42},
   note={Translated from the second French edition by Leonard L. Scott},
   publisher={Springer-Verlag, New York-Heidelberg},
   date={1977},
   pages={x+170},
   isbn={0-387-90190-6},
   review={\MR{0450380}},
}


\bib{TW95}{article}{
   author={Th\'{e}venaz, Jacques},
   author={Webb, Peter},
   title={The structure of Mackey functors},
   journal={Trans. Amer. Math. Soc.},
   volume={347},
   date={1995},
   number={6},
   pages={1865--1961},
   issn={0002-9947},
   review={\MR{1261590}},
   doi={10.2307/2154915},
}
	

\bib{Thev}{article}{
   author={Th\'{e}venaz, Jacques},
   title={Some remarks on $G$-functors and the Brauer morphism},
   journal={J. Reine Angew. Math.},
   volume={384},
   date={1988},
   pages={24--56},
   issn={0075-4102},
   review={\MR{929977}},
   doi={10.1515/crll.1988.384.24},
}

\bib{zhang}{article}{
    AUTHOR = {Zhang, Ningchuan},
     TITLE = {Analogs of {D}irichlet {$L$}-functions in chromatic homotopy
              theory},
   JOURNAL = {Adv. Math.},
  FJOURNAL = {Advances in Mathematics},
    VOLUME = {399},
      YEAR = {2022},
     PAGES = {Paper No. 108267, 84},
      ISSN = {0001-8708},
   MRCLASS = {55P42 (11S40 55N22 55Q50)},
  MRNUMBER = {4384614},
MRREVIEWER = {Guchuan Li},
       DOI = {10.1016/j.aim.2022.108267},
       URL = {https://doi-org.ezproxy.uky.edu/10.1016/j.aim.2022.108267},
}
	

\end{biblist}
\end{bibdiv}
\end{document}